\pdfoutput=1

\documentclass[12pt]{amsart}

\usepackage{amssymb, a4, mathrsfs}

\newtheorem{them}{Theorem}[section]

\newtheorem{prop}[them]{Proposition}
\newtheorem{lem}[them]{Lemma}

\newtheorem{ex}[them]{Example}

\theoremstyle{definition}
\newtheorem{defn}[them]{Definition}
\usepackage{pstricks,pst-node,pst-plot}
\usepackage{tikz}
\usepackage{verbatim}
\usepackage{pgf}
\usepackage{pifont}
\usepackage{bbding}
\usetikzlibrary{arrows,automata}
\usepackage[latin1]{inputenc}
\usepackage{graphicx,amscd}
\usepackage{CJK}
\usepackage{indentfirst}
\theoremstyle{definition}

\numberwithin{them}{section}
\numberwithin{equation}{section}

\begin{document}
\title{Quasi-semilattices on networks}
\thanks{Corresponding author: Dazhi Meng}
\thanks{The  first author was supported  by NSFC (Grant No:11501331)}
\author{Yanhui Wang}

\address{College of Mathematics and Systems Science, Shangdong University of Science and Technology, Qingdao, 266590, P. R. China}
\email{yanhuiwang@sdust.edu.cn}
\author{Dazhi Meng}
\address{ College of Applied Science, Beijing University of Technology, 100 Pingleyuan, Beijing, 10024, P. R. China}
\email{dzhmeng07@163.com}

\maketitle
\renewcommand{\thefootnote}{\empty}

\begin{abstract}
This paper introduces the tensor representation of a network, here tensors are the primitive structures of the network. In view of tensor chains, two binary operations on tensor sets are defined: chain addition and reducing. Based on the reducing operation, the tensor chain representation of subnetworks of a network is given, and it is proved that all connected subnetworks of a network (here refers to the tensor chain generated by primitive structures) form a quasi-semilattice with respect to reducing, namely {\it network quasi-semilattices}. Here, quasi-semilattices refer to algebraic systems that are idempotent commutative and do not satisfy the association law. Then, we discuss the subalgebra structures of the network quasi-semilattice in terms of two equivalent relations $\sigma$ and $\delta$. $\delta$ is a congruence. Each $\delta$-class forms a semilattice with respect to reducing, that is, an idempotent commutative semigroup, and also each $\delta$-class has an order structure with the maximum element and minimum elements. Here, the minimum elements correspond to the spanning tree in graph theory. Finally, we discuss how three path algebras: graph inverse semigroups, Leavitt path algebra and Cuntz-Krieger graph $C^*$-algebra are constructed in terms of tensors with respect to chain-addition.
\end{abstract}

\section*{Introduction}

The research of network algebra in theory and application has attracted extensive attentions\cite{G1}. For example, the usual adjacency matrix  is a typical algebraic representation of networks. According to the attribute classification of the association among vertices in a network, the most common networks are correlation networks and simple directed networks. A correlation network refers to the undirected network in which vertices are associated through correlations.  An undirected network in which the connection  is one vertex to another one, is called a simple correlation network. The connection in a simple directed network is represented by a directed arrow pointing from one vertex to another. Such a network is also known as a simple logic network (or a first-order logic network), and so on. In 2004, Bowers et al. \cite{Peter2} introduced the concept and algorithm of high-order logic relationships and pointed out its valuable application in protein researches. This kind of network with high-order logic is much more complex than simple directed networks. It reflects the complex directed interaction between multiple vertices. However, this high-order logic network is difficult to be represented by adjacency matrix, which puts forward a new topic on the algebraic representation of networks. The tensor representation of networks firstly introduced in this paper is a generalization of adjacency matrices, which is suitable for low-order and high-order correlation networks and logic networks. In this paper, correlation networks and high-order logic networks extended by the research in \cite{Peter2} are collectively referred to as {\it general networks}, which are the main research object of this paper.

The network  model of a system is usually represented by the geometric representation of graph theory, that is, the network vertices and the connecting lines between vertices. This is a view that vertices are regarded as research subjects (primitives), and the connection between vertices is regarded as the correlation attribute of vertices. Such a representation has achieved rich results in network research \cite{Newman3,Boccaletti4,Newman5}. Another perspective to observe and study networks is to take the complex association structure between vertices as the research objects, such as the high-order logic relationships. In this paper, the basic association between vertices is represented by tensor, and these tensors are regarded as the basic structural elements of a network, also known as {\it primitives}, and the vertices are regarded as the association medium between primitives. The view of this paper is inspired by the thought of structuralism \cite{Jean6}. The research of networks  based on structuralism will focus on the algebra of network structural elements, that is, to study the algebra consisting of  primitives, especially to explore the subalgebra structure of this kind of network algebra. This is an effective method to study the network focusing on the coupling functions and attribute of a system.

In view of the complexity of association structures, the perspective of structuralism needs a new representation method of  networks, that is, the tensor representation   given in Section~\ref{sec:TensorRep}, which is a typical algebraic representation of general networks. A special case of this representation is the adjacency matrix in correlation networks and simple directed networks. The adjacency matrix of a network plays an indispensable role in the research of networks. The stoichiometric matrix is an adjacency matrix, which represents  chemical reaction networks such as the metabolic network in biology, and plays an important role in the research of biochemistry \cite{Bernhard7}. Notice that the adjacency matrix only can represent the linear network which means that the connection is one to one . While  high-order logic networks are essentially nonlinear \cite{Peter2}. Tensor representation can well describe more general complex networks including  high-order logic networks, and this algebraic representation introduces a completely algebraic algorithm to network calculation. Although the pure algebraic algorithm loses the intuition of graph theory, this more abstract  model of networks will bring new significance in theory and application. The tensor representation of networks will obviously lead to the study of related algebra. As we all know, network algebra caused by algebraic methods has become an encouraging research direction in a certain range\cite{G1}. It shows the important value of algebraic methods in network researches.

The structure of this paper is as follows. Section~\ref{sec:TensorRep} presents the tensor representation of primitives of a network. In Section~\ref{sec:quasi-semilattice} two binary operations, called {\it chain addition of tensors} and {\it reducing of 2-chains}, are introduced. The set of all 2-chains of a network $\Gamma$ with respect to  reducing generates a quasi-semilattice ${\mathcal{L}}(\Gamma)$ in which every element is idempotent and the binary operation is commutative but non-associative. In Section~\ref{sec:subalgebra}, the subalgebra of ${\mathcal{L}}(\Gamma)$ is introduced based on two equivalent relations $\sigma$ and $\delta$. We show that $\delta$ is a congruence and  each $\delta$-class is a semilattice which is an idempotent commutative semigroup. Further the relationship between two $\delta$-classes is discussed and give a condition when two $\delta$-classes are isomorphic. Section~\ref{sec:orderrelation}, a partial order relation is given and the local maximum and the minimum elements of $\delta$-classes are investigated. Section~\ref{sec:path} claims that the graph inverse semigroups\cite{Ash8}, Leavitt path algebra\cite{Kumjian9} and Cuntz-Krieger graph $C^*$-algebra\cite{Abrams10, Ara11} can be generated by tensor representation of a one-order logic network. The three path algebras consist of only paths of one-order logic network which  are the special subgraphs of a network, while  the elements in ${\mathcal{L}}(\Gamma)$ are not only the paths but also other subnetworks such as branchings.

This paper introduces the tensor representation of a network, which is different from the traditional network research method taking vertices as the research object. The research on the algebraic system generated by tensor is applicable to related networks, logic networks and hybrid networks. Tensors generates not only the path of the network, but also all connected subnetworks of the network (such as branches, especially in high-order logic networks. The traditional path cannot represent all of the high-order subnetworks. Therefore, the tensor representation of networks not only introduces new algebraic theory, but also provides a new method for the study of networks.

\section{Tensor representation of a general network}\label{sec:TensorRep}

\subsection{General networks}
The mathematical model of a network consists of vertices and their associated relationships. Therefore, the network can be divided into three types according to the associated relationships: correlation networks, logic networks and hybrid networks. A network in which all the associated relationships are logically related is called a {\it logic network}, also called a {\it directed network} because its relationships are represented by directed arrows in graph theory. Similarly, the network consists of correlations is called a {\it correlation network},  and also called an {\it undirected network} in graph theory. Then the concept of hybrid networks is obvious. In this paper, correlation networks, logic networks and hybrid networks are collectively referred to as "{\it general network}``.

In 2004, Bowers et al. introduced the concept of high-order logic  relationships by using Boolean logic relations and information entropy \cite{Peter2}, which extends the traditional one to one directed correlation. Logic networks are divided into first-order and high-order. The logic relationship between two vertices $i$ and $j$ is represented by directed arrow $``i \bullet\longrightarrow\bullet j"$. Such a network is called a {\it first-order logic relationship}. The  {\it high-order logic relationship} is represented by the directed arrow starting at $n$ vertices and pointing to $m$ vertices after $n$ vertices converge. In particular, a high-order logic relationship which starting at $n$ vertices and pointing to one vertex after $n$ vertices converse is called a {\it $n$-th logic relationship}.  For example, the second-order logic refers to the logic that two vertices jointly affect one vertex, that is,
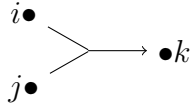
\begin{center}
\begin{tikzpicture}
\path (0,0)    node (A){$i\bullet$}
      (0,-1)  node (B) {$j\bullet$}
      (1,-0.57) node(D){}
      (1,-0.43) node(F){}
      (0.71,-0.5) node (E){}
      (2,-0.5) node (C) {$\bullet k$};
\draw[-,black] (A)--(D);
\draw[-,black] (B)--(F);
\draw[->,black] (E)--(C);
\end{tikzpicture}.
\end{center}

\subsection{Tensor representation of networks}\label{subsec:TensorRep}
As we all know, both first-order correlation networks and first-order logic networks can be represented by adjacency matrix (as well as first-order hybrid networks), which is equivalent to a graph representation in graph theory. Matrix representation has a wide range of important applications \cite{Bernhard7}, which highlights the significance of algebraic representation. However, it is hard to generalise the adjacency  matrices to high-order logic networks as it is difficult to represent high-dimensional geometric space images by matrix. In order to realize the unified algebraic representation of general networks, tensor representation is introduced in this paper as follows.

 The second-order covariant tensor $T_{ij}$ is used to represent the correlation between vertices $i$ and $j$, painted as $``i \bullet\rule[3.5pt]{0.5cm}{0.05em}\bullet j"$; the first-order covariant contra-variant tensor $T_{i}^{j}$ is used to represent the first-order logical relationship from vertex $i$ to vertex $j$, painted as $``i \bullet\longrightarrow\bullet j"$; the tensor with $n$ covariant indexes and $m$ contra-variant indexes $T_{i_1i_2\ldots i_n}^{j_1j_2\ldots j_m}$ is used to represent the high-order logic relationship from $n$ vertices to $m$ vertices.
Thus, we complete the tensor representation of the connection structures in general networks.

In tensor representation, the superscript is called {\it contra-variant index}, and the subscript is called {\it covariant index}, where the covariant index is symmetric and disordered. The union of covariant index and contra-variant index is called {\it tensor index set}. In order to present the contra-variant index, the covariant index and the index of a tensor conveniently, three maps from the set $T$ of tensors of a network to the power set ${\mathcal{P}}(\Gamma^0)$ of the set $\Gamma^0$ of vertices of a network $\Gamma$ are defined as follows. $\emptyset$ denotes the empty set. For arbitrary $A, B \subset \Gamma^0$ and $A \neq \emptyset$,
\[\psi(T_A^B)=A,\hspace{2mm} \varphi(T_A^B)=B\ \mbox{and}\ I(T_A^B)=A \cup B.\]

\section{Quasi-semilattices}\label{sec:quasi-semilattice}

The tensor representation presents the most basic association between vertices in a network, and the set consisting of tensors is the primitive structure of a network. This section studies algebraic systems generated by network primitive structures. For ease of description, let $\Gamma$ denote a general network, $\Gamma^0$ denote a set of vertices of $\Gamma$, and $T$ denote a set of tensors of $\Gamma$ containing the empty relationship $\emptyset$.  In this paper, we only consider finite networks, that is, $|\Gamma^0|<\infty$ and $|T|<\infty$. Suppose that $|T|=n$, $n \in {\mathbb{N}}$. Sorting elements in $T$, we get $T=\{t_1, t_2, \ldots, t_n\}$. In a network, two different tensors are connected through common vertices.  Therefore, a binary operation of  two tensors is defined here, which is called the {\it chain addition of tensors}, written as $\oplus$. The correlation between two tensors is generated by chain addition.

\begin{defn}\label{ChainAdd}
For any $t, t_1, t_2 \in T$,
$$t_1 \oplus t_2 =
\begin{cases}
\ \oplus(t_1,t_2) & \mbox{if } \ t_1\neq t_2 \ \mbox{and}\ I(t_1) \cap I(t_2)\neq \emptyset \\
\ t_1(\mbox{or}\ t_2) & \mbox{if}\ t_1=t_2\\
\ \emptyset & \mbox{otherwise,}
\end{cases}
$$
$$t \oplus \emptyset =\emptyset \oplus t=\emptyset, $$
where $(t_1,t_2)$ denotes the set $\{t_1,t_2\}$ in $\oplus (t_1,t_2)$, and $\oplus(t_1,t_2)$ denotes the subnetwork obtained by combining  all the same indexes of $t_1$ and $t_2$  .
\end{defn}

Here we should stress that in Definition~\ref{ChainAdd} $\oplus(t_1, t_2)=\oplus(t_2, t_1)$ as $(t_1, t_2)$ denotes the set whose elements are $t_1$ and $t_2$. We call $\oplus(t_1, t_2)$ a {\it 2-chain}.

In a simple correlation network, there exists only one type of 2-chain, that is, when two different tensors $T_{i_1j_1}$ and  $T_{i_2j_2}$ have a common covariant index, a 2-chain is formed, that is,
$``j_1 \bullet\rule[3.5pt]{0.5cm}{0.05em}\bullet i_1 \rule[3.5pt]{0.5cm}{0.05em}j_2"$ if $i_1=i_2$. In a high-order logic network only consisting of  first-order and second-order logic relationships, the formation of 2-chains is divided into three categories: $\rm (C1)$ two first-order logic relationships, $\rm (C2)$ one first-order and one second-order logic relationships , and $\rm(C3)$ two second-order logic relationships. For Category $\rm (C1)$, according to the combination of the same index among the contra-variant index and covariant index of two different first-order logic relationships $T_{i_1}^{k_1}$ and $T_{i_2}^{k_2}$, there are three types of 2-chains:

$\rm (i)$\hspace{2mm} $i_1 \bullet\longrightarrow\bullet k_1\longrightarrow\bullet k_2$ if $k_1=i_2$;

$\rm (ii)$\hspace{2mm}  $k_1 \bullet\longleftarrow\bullet i_1\longrightarrow\bullet k_2$ if $i_1=i_2$;

$\rm (iii)$\hspace{2mm} $i_1 \bullet\longrightarrow\bullet k_1\longleftarrow\bullet i_2$ if $k_1=k_2$.\\
For Category $\rm (C2)$ and $\rm (C3)$ there are 5 and 7 types, respectively, according to the number and combination of the same indexes of tensors. Fig. 1 is an example which shows three types for Category $\rm (C3)$. In hybrid networks, besides the 2-chain types generated by two correlation or logic relationships, there is also a 2-chain type formed by one correlation and one logic relationships.

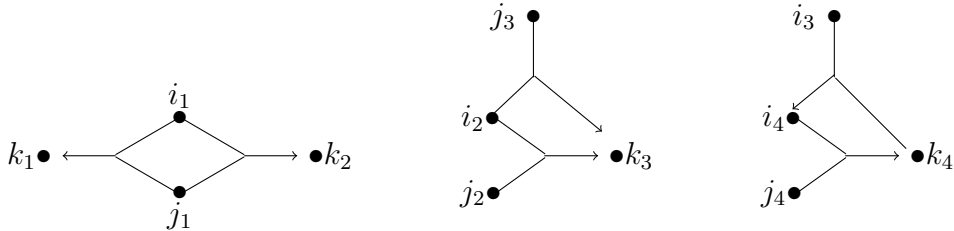
\begin{center}
\begin{tikzpicture}
\path (-1,-0.57) node(G){}
      (-1,-0.43) node(I){}
      (-0.71,-0.5) node (H){}
      (-2,-0.5) node (J) {$k_1 \bullet $}
      (0,0.3)    node (A0){$i_1$}
      (0,0)    node (A){$\bullet$}
       (0.15,0.1)    node (A2){}
       (-0.15,0.1)    node (A3){}
      (0,-1)  node (B) {$\bullet$}
      (0,-1.3)  node (B0) {$j_1$}
      (0.15,-1.1)  node (B2) {}
      (-0.15,-1.1)  node (B3) {}
      (1,-0.57) node(D){}
      (1,-0.43) node(F){}
      (0.71,-0.5) node (E){}
      (2,-0.5) node (C) {$\bullet k_2$}
      (4,0)    node (A4){$i_2\bullet$}
      (4,0.14)    node (A5){}
      (4,-0.1)    node (A6){}
      (4,-1)  node (B4) {$j_2\bullet$}
      (4,-1.14)  node (B5) {}
      (4.3,1.35)  node (L4) {$j_3$}
      (4.7,1.35)  node (L5) {$\bullet$}
      (4.7,1.45)  node (L6) {}
      (4.7,0.4)  node (E2) {}
       (4.85,0.7)  node (E3) {}
       (4.57,0.68)  node (E4) {}
      (5,-0.58) node(D1){}
      (5,-0.42) node(F1){}
      (4.7,-0.5) node (E1){}
      (6,-0.5) node (C1) {$\bullet k_3$}
      (8,0)    node (X4){$i_4\bullet$}
      (8,0.14)    node (X5){}
      (8,0)    node (X6){}
      (8,-1)  node (Y4) {$j_4\bullet$}
      (8,-1.14)  node (Y5) {}
      (8.3,1.35)  node (W4) {$i_3$}
      (8.7,1.35)  node (W5) {$\bullet$}
      (8.7,1.45)  node (W6) {}
      (8.7,0.4)  node (P2) {}
       (8.85,0.7)  node (P3) {}
       (8.57,0.68)  node (P4) {}
      (9,-0.58) node(D2){}
      (9,-0.42) node(F2){}
      (8.7,-0.5) node (P1){}
      (10,-0.5) node (C2) {$\bullet k_4$}
      (9.8,-0.55) node (C3){};
 \draw[-,black] (W6)--(P2);
 \draw[->,black] (P3)--(X6);
\draw[-,black] (P4)--(C3);
\draw[-,black] (X5)--(D2);
\draw[-,black] (Y5)--(F2);
\draw[->,black] (P1)--(C2);
 \draw[-,black] (L6)--(E2);
 \draw[-,black] (A6)--(E3);
\draw[->,black] (E4)--(C1);
\draw[-,black] (A5)--(D1);
\draw[-,black] (B5)--(F1);
\draw[->,black] (E1)--(C1);
\draw[-,black] (A3)--(D);
\draw[-,black] (B3)--(F);
\draw[->,black] (E)--(C);
\draw[-,black] (A2)--(G);
\draw[-,black] (B2)--(I);
\draw[->,black] (H)--(J);
\end{tikzpicture}

Fig.1 Three types for Category $\rm (C3)$.
\end{center}
\noindent From another perspective of the network, if the tensor is regarded as a vertex, two tensors are connected by the common vertices, which is independent of the type of tensors. In this way, the algebraic system generated by tensor chain addition is consistent for correlation networks, (high-order) logic networks and hybrid networks.

For any two different tensors  $t_1, t_2 \in T$ with common indexes,  a non-empty 2-chain is generated by chain addition $\oplus(t_1, t_2)$, and then if there exists $t_3 \in T$ such that $\oplus(t_1, t_3)$ is a non-empty 2-chain, then it is obvious that on the basis of 2-chain $\oplus(t_1, t_2)$, a subnetwork consisting of $t_1, t_2$ and $t_3$ is obtained by doing the chain addition  of $t_1$ and $t_3$. Such a subnetwork is expressed as $(t_1 \oplus t_2)\oplus t_3 = \oplus(t_1, t_2, t_3)$. On the other hand,  since 2-chains $\oplus(t_1, t_2)$ and $\oplus(t_1,t_3)$ have a common tensor $t_1$  it is obvious that a subnetwork consisting of $t_1, t_2$ and $t_3$ is obtained by reducing  the same tensor $t_1$ in the two 2-chains to one. Such a subnetwork is expressed as $(t_1 \oplus t_2)\oplus (t_1 \oplus t_3 )= \oplus(t_1, t_2, t_3)$. Obviously, the two subnetworks are the same. Notice that it is possible that both $\oplus(t_1, t_3)$ and $\oplus(t_2,t_3)$ are 2-chains. In this case, the former method does not specify whether $t_1$ and $t_3$ are  chain added or $t_2$ and $t_3$ are  chain added, so the result is not unique. The latter method reduces the same tensor to one to obtain a subnetwork while maintaining the 2-chain relationships, and such a subnetwork is unique. Therefore,  in the following the latter method is taken to generate subnetworks from the perspective of 2-chain reducing.

For any  $t_i \in T$,  we write $t_i = \oplus(t_i, t_i)$ and simply denoted by $l_{ii}$. For any two different tensors $t_i, t_j \in T$, if $t_i \oplus t_j = \oplus(t_i, t_j) \neq \emptyset$ then $\oplus(t_i, t_j)$ is simply written as $l_{ij}$, where the subscripts $i$ and $j$ of $l_{ij}$ are unordered, that is, $l_{ij} = l_{ji}$. Set $N = \{1, 2, \ldots, n\}$. Let $T_{n2}$ denotes the set of all non-empty 2-chains generated by all tensors in $T$, that is $T_{n2} = \{l_{ij}| i, j = 1, 2, \ldots, n; i \neq j; l_{ij} \neq \emptyset\}$. The {\it reducing} operation is defined in the following:

\begin{defn}\label{two-chain reducing}
Let $T_{n2}$ be the set of all non-empty 2-chains generated by all tensors in $T$. For any $l_{ii}= \oplus(t_i, t_i)$, $l_{jj}= \oplus(t_j, t_j)$, $l_{i_1j_1} = \oplus(t_{i_1}, t_{j_1})$ and $l_{i_2j_2} = \oplus(t_{i_2}, t_{j_2})$ in $T_{2n}$, we define a binary operation $\mathring{\cup}$, called {\it reducing}, by the rule that
$$l_{i_1j_1} \mathring{\cup} l_{i_2j_2}  =
\begin{cases}
\ \oplus(t_1, t_2, t_3)|\{t_1, t_2, t_3\}=\{t_{i_1},t_{j_1},t_{i_2},t_{j_2}\} & \mbox{if } \ l_{i_1j_1} \neq l_{i_2j_2}\\
 \ \ &\hspace{4mm}\mbox{and}\ (t_{i_1},t_{j_1}) \cap (t_{i_2},t_{j_2})\neq \emptyset \\
\ l_{i_1j_1}(\mbox{or} \  l_{i_2j_2} ) & \mbox{if} \ l_{i_1j_1} = l_{i_2j_2}\\
\ \emptyset & \mbox{otherwise,}
\end{cases}
$$
where $\oplus(t_1, t_2, t_3)$ denotes the subnetwork obtained by the chain addition of $l_{i_1j_1}$ and $l_{i_2j_2}$,  $\{t_1, t_2, t_3\}=\{t_{i_1},t_{j_1},t_{i_2},t_{j_2}\} $ means that $\{t_1, t_2, t_3\}$ is the set obtained by combining the same tensor as one in the set $\{t_{i_1},t_{j_1},t_{i_2},t_{j_2}\}$;
$$l_{ii} \mathring{\cup} l_{i_1j_1}  =
\begin{cases}
\ l_{i_1j_1} & \mbox{if } t_i \in (t_{i_1},t_{j_1}) \\
\ \emptyset & \mbox{otherwise,}
\end{cases}
$$
$$l_{ii} \mathring{\cup} l_{jj}  =
\begin{cases}
\ l_{ii} (\mbox{or} \ l_{jj})& \mbox{if } l_{ii}=l_{jj} \\
\ \emptyset & \mbox{otherwise.}
\end{cases}
$$
\end{defn}
It is natural to generalise Definition~\ref{two-chain reducing} to a set $S$ of 2-chains by reducing the same tensors of 2-chains in $S$ into one. This is a method to treat the arbitrarily set of 2-chains as a specific base and generate a specific chain through the reducing. Obviously, a set of specific  primitive structures corresponds to a specific chain one by one. A specific chain is a subnetwork. Therefore, any non empty subset of $T_{2n}$ of the network $\Gamma$ can generate a subnet of $\Gamma$ by reducing, and then all non empty subsets of $T_{2n}$ can generate all subnetworks of $\Gamma$ by reducing. Let $S$ be a non-empty subset of $\Gamma_{2n}$ and let $P$ be the set of all tensors forming 2-chains in $S$. $\oplus(P|S)$ denotes the network generated by tensors in $P$ with respect to 2-chains in $S$, that is , the network generated by reducing 2-chains in $S$.  So $\mathring{\cup}S = \oplus(P|S)$.

\begin{ex}\label{ex1}
 Let $S = \{t_{12}, t_{23}, t_{24}\}$ and $P = \{t_1, t_2, t_3, t_4\}$. $\oplus(P|S)$ is as follows.

\vspace{12mm}\setlength {\unitlength}{0.5cm}
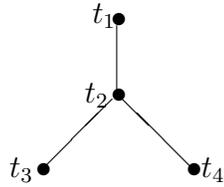
\begin{picture}(27,7)
\put(8,3){$\bullet$} \put(7.3,3){$t_3$}
\put(10,7){$\bullet$}\put(9.5,7.1){$t_1$}
\put(10,5){$\bullet$}\put(9.3,5){$t_2$}
\put(12,3){$\bullet$}\put(12.4,3){$t_4$}

\put(10.15,5.3){\line(0,1){1.75}}
\put(8.25,3.3){\line(1,1){1.78}}
\put(12.1,3.3){\line(-1,1){1.8}}
\end{picture}
\end{ex}
\begin{defn}\label{connected}
Given a subnetwork $\oplus(P|S)$, if for any two different tensors $t_i, t_j \in P$, there exist $t_{k_1}, t_{k_2}, \ldots, t_{k_m} \in P$ such that $$ t_i = t_{k_1}, \ \oplus(t_{k_{\ell}}, t_{k_{\ell+1}}) \in S,\ t_{k_m} = t_j,$$
where $\ell = 1, 2, \ldots, m-1$, then $\oplus(P|S)$ is called a {\it connected subnetwork}.
\end{defn}

In the following, a connected subnetwork $\oplus(P|S)$ is referred as a {\it tensor chain}. $\oplus(P|S)$ is called a {\it $k$-chain} if $|P|=k$. In particular, tensors are 1-chains. In this paper, we only consider tensor chains.

Here we stress that as shown in Example~\ref{ex1}, a tensor chain is not necessarily a path as mentioned in traditional graph theory. It is easy to know that any tensor chain can be generated by reducing the same tensors of 2-chains. Obviously, there is a one to one correspondence between tensor chains and the set of 2-chains. So a tensor chain is represented by a set of 2-chains in this paper. To emphasize the tensors in a tensor chain, the tensor chain is usually written as $\oplus(P|S)$, where $S$ denotes the set of 2-chains and $P$ denotes the set of all tensors generating  2-chains belonging to $S$.

In the research of networks, compared with 2-chain's structures, the research of general tensor chains is more macro and rich. Therefore, in the following we focuses on the algebraic structure of the system consisting of tensor chains. Obviously, the new subnetwork obtained by joining two connected subnetworks  is still connected, so the reducing of the set of 2-chains can be extended to any two tensor chains. The following defines the reducing between two tensor chains.
\begin{defn}\label{Ladd}
For any two non-empty tensor chains $l_1 = \oplus(P_1|S_1)$  and $l_2 = \oplus(P_2|S_2) $,
$$l_1 \mathring{\cup} l_2 =
\begin{cases}
\ \mathring{\cup}(S_1 \cup S_2) & \mbox{if } \ P_1 \cap P_2 \neq \emptyset \\
\ \emptyset & \mbox{otherwise.}
\end{cases}
$$
\end{defn}

In \cite{Hala12}, as a generalization of groups, an algebraic system with a binary operation that does not satisfy the association law is discussed. This kind of algebraic system has an identity, and each element has  an inverse, but does not satisfy the association law, which is called {\it quasi-group}. Obviously, whether the associative law is satisfied or not is the fundamental difference between groups and quasi-groups. Similarly, as a generalization of semilattices, an algebraic system that do not satisfy the association law is called {\it quasi-semilattice} if it is idempotent commutative with respect to the binary operation. Naturally, whether the associative law is satisfied or not is the fundamental difference between semilattices and quasi-semilattices. The quasi-semilattice defined here is called {\it commutative idempotent magam} in~\cite{http13}. An example of quasi- semilattice with three elements is given in~\cite{http13} as follows.

\begin{center}
\begin{tabular}{c|ccc}
      $*$& $a$& $b$&$c$\\
\hline  $a$&$a$&$c$&$b$\\
  $b$&$c$&$b$&$a$\\
  $c$&$b$&$a$&$c$
\end{tabular}
\end{center}

The set of all non empty tensor chains generated by tensors in $T$  is called {\it the universal set of tensor chains of the network $\Gamma$}, which is recorded as $\oplus_{\mbox{all}}(t_1, t_2, \ldots, t_n)$. For convenience, ${\mathcal{L}}(\Gamma)$ denotes the union of $\oplus_{\mbox{all}}(t_1, t_2, \ldots, t_n)$ and  $\{\emptyset\}$ where $\emptyset$ is the empty relation, that is, ${\mathcal{L}(\Gamma)} = \oplus_{\mbox{all}}(t_1, t_2, \ldots, t_n) \cup \{\emptyset\}$. In the following, we show that $\mathcal{L}(\Gamma)$ forms a quasi-semilattice with respect to  $\mathring{\cup}$.

\begin{them}
For any general network $\Gamma$, $\mathcal{L}(\Gamma)$ forms a quasi-semilattice with respect to  $\mathring{\cup}$, called a {\it network quasi-semilattice}.
\end{them}
\begin{proof}
 It is easy to see that for any two elements $l_1, l_2 \in {\mathcal{L}}(\Gamma)$, we have $l_1 \mathring{\cup} l_2 \in {\mathcal{L}}(\Gamma)$, that is, ${\mathcal{L}}(\Gamma)$ is closed with respect to $\mathring{\cup}$. Clearly, $\emptyset \oplus \emptyset = \emptyset$ and for any non-empty tensor chain $l = \oplus(P|S)$, we have $l \mathring{\cup} l = \mathring{\cup} (S \cup S)= \mathring{\cup} S = l$. So every element of ${\mathcal{L}}(\Gamma)$ is idempotent. Naturally, $\mathring{\cup}$ is commutative as  the union of sets is commutative.

Now we show that the associative law does not hold.  Suppose that $l_1 = \oplus(P_1|S_1)$, $l_2 = \oplus(P_2|S_2)$ and $l_3 = \oplus(P_3|S_3)$ are three non-empty elements of ${\mathcal{L}}(\Gamma)$ such that $P_1 \cap P_2 \neq \emptyset$, $P_1 \cap P_3 \neq \emptyset$ and $P_2 \cap P_3 = \emptyset$. Then we have $l_1 \mathring{\cup} l_2 = \mathring{\cup}(S_1 \cup S_2) \neq \emptyset$ and $l_2 \mathring{\cup} l_3 = \emptyset$. Since $P_1 \cap P_3 \neq \emptyset$, it follows that $(P_1 \cup P_2) \cap P_3 \neq \emptyset$. Further, $(l_1 \mathring{\cup} l_2 ) \mathring{\cup} l_3 = \mathring{\cup}(S_1 \cup S_2) \mathring{\cup} l_3 = \mathring{\cup}(S_1 \cup S_2 \cup S_3) \neq \emptyset$  and $l_1 \mathring{\cup} (l_2 \mathring{\cup} l_3) =l_1 \mathring{\cup} \emptyset =\emptyset $, which implies that  $(l_1 \mathring{\cup} l_2 ) \mathring{\cup} l_3 \neq l_1 \mathring{\cup} (l_2 \mathring{\cup} l_3)$.
\end{proof}

\section{Subalgebra of  network quasi-semilattice ${\mathcal{L}}(\Gamma)$}\label{sec:subalgebra}

The subalgebra structure of network quasi-semilattice ${\mathcal{L}}(\Gamma)$ is studied by means of equivalence relations and congruences on ${\mathcal{L}}(\Gamma)$.

\subsection{Congruences}

 Certainly there exists two natural ways to classify tensor chains. One is by the number of tensors generating a chain, the other is the tensors generating a chain. According to the two classifying ways we define two relations $\sigma$ and $\delta$ on ${\mathcal{L}}(\Gamma)$ as follows.

\begin{defn}
For any $l_1 = \oplus(P_1|S_1)$ and $l_2 = \oplus(P_2|S_2)$,

$\rm (i)$ $(l_1, l_2) \in \sigma$ if $|P_1| = |P_2|$;

$\rm (ii)$ $(l_1, l_2) \in \delta$ if $P_1 = P_2$.
\end{defn}

The following lemma presents that $\sigma$ and $\delta$ are equivalences.

\begin{lem}\label{TwoEquivalence}
Relations $\sigma$ and $\delta$ are equivalent relations.
\end{lem}
 \begin{proof}
 Suppose that $l_1 = \oplus(P_1|S_1)$, $l_2 = \oplus(P_2|S_2)$ and $l_3 = \oplus(P_3|S_3)$ are elements of ${\mathcal{L}}(\Gamma)$.

 Now we show that $\sigma$ is equivalent. $\rm (i)$ For any $l = \oplus(P|S) \in {\mathcal{L}}(\Gamma)$, we have $(l, l) \in \sigma$ as $|P| = |P|$.  Thus $\sigma$ is reflexive. $\rm (ii)$ If $(l_1, l_2) \in \sigma$ then $|P_1| = |P_2|$ which implies that $(l_2, l_1) \in \sigma$. So $\sigma$ is symmetric. $\rm (iii)$ If $(l_1, l_2) \in \sigma$ and $(l_2, l_3) \in \sigma$ then $|P_1| = |P_2|$ and $|P_2| = |P_3|$, so we have $|P_1| = |P_3|$ which follows that $(l_1, l_3) \in \sigma$. Thus $\sigma$ is transitive.

 It is to see that $\delta$ is reflexive as for any $l = \oplus(P|S) \in {\mathcal{L}}(\Gamma)$, we have $(l, l) \in \delta$ as $P = P$. If $(l_1, l_2) \in \delta$ then $P_1 = P_2$ which follows that $(l_2, l_1) \in \delta$ and so $\delta$ is symmetric. Suppose that $(l_1, l_2) \in \delta$ and $(l_2, l_3) \in \delta$ then $P_1 = P_2$ and $P_2 = P_3$ which implies that $P_1 = P_3$ and so $\delta$ is transitive.
 \end{proof}

An equivalence on an algebraic system corresponds to a partition on the object set of the  system which is independent of the operation of algebraic system. An equivalence $\rho$ on an algebraic system $(A, *)$ is a congruence with respect to  $*$ if $\rho$ is left and right compatible with respect to $*$, that is, for any $a, b, c \in A$, if $(a, b) \in \rho$ then $(c * a, c * b) \in \rho$ and $(a * c, b * c) \in \rho$. If the operation is commutative, then a left compatible relation is also a right compatible.

In $({\mathcal{L}}(\Gamma), \mathring{\cup})$, $\sigma$ is not a congruence as it is not compatible. Clearly, $\mathring{\cup}$ is commutative so it is sufficient to show that $\sigma$ is not left compatible with respect to $\mathring{\cup}$. Let $l_1, l_2, l_3 \in {\mathcal{L}}(\Gamma)$ be such that $(l_1, l_2) \in \sigma$, $l_3 \mathring{\cup} l_1 = \emptyset$ and $l_3 \mathring{\cup} l_2 \neq \emptyset$. As $(\emptyset, l_3 \mathring{\cup} l_2) \notin \sigma$ we get $(l_3 \mathring{\cup} l_1, l_3 \mathring{\cup} l_2) \notin \sigma$ which follows that $\sigma$ is not left compatible.

\begin{lem}
Relation $\delta$ is a congruence on ${\mathcal{L}}(\Gamma)$.
\end{lem}
\begin{proof}
It is sufficient to show that $\delta$ is left compatible since $\delta$ is an equivalence by Lemma~\ref{TwoEquivalence} and $\mathring{\cup}$ is commutative. Let $l_1 = \oplus(P_1|S_1)$, $L_2 = \oplus(P_2|S_2)$ and $l_3 = \oplus(P_3|S_3)$ be elements of ${\mathcal{L}}(\Gamma)$ such that $(l_1, l_2) \in \delta$. So we have $P_1 = P_2$ which follows that $P_1 \cap P_3 = P_2 \cap P_3$. if $P_1 \cap P_3 = \emptyset$ we have $(l_3 \mathring{\cup} l_1, l_3 \mathring{\cup} l_2) \in \delta$ as $l_3 \mathring{\cup} l_1 = l_3 \mathring{\cup} l_2 = \emptyset$. While if $P_1 \cap P_3 \neq \emptyset$, we get $l_3 \mathring{\cup} l_1 = \mathring{\cup}(S_1 \cup S_3) = \oplus(P_1 \cup P_3|S_1 \cup S_3)$ and $l_3 \mathring{\cup} l_2 = \mathring{\cup}(S_2 \cup S_3) = \oplus(P_2 \cup P_3|S_2 \cup S_3)$. It follows from $P_1 = P_2$ that $P_1 \cup P_3 = P_2 \cup P_3$ and so $(l_3 \mathring{\cup} l_1, l_3 \mathring{\cup} l_2) \in \delta$.
\end{proof}

Let $l_1 = \oplus(P_1|S_1)$ and $l_2 = \oplus(P_2|S_2)$ be elements of ${\mathcal{L}}(\Gamma)$ such that $|P_1| = |P_2| = k \in \mathbb{N}$. If $|P_1 \cap P_2| = m (0< m <k)$, then we obtain a $2k-m$-chain $l_1 \mathring{\cup} l_2$. If $P_1 \cap P_2 = \emptyset$, we get $l_1 \mathring{\cup} l_2 = \emptyset$. So in general $\mathring{\cup}$ is impossible closed in a $\sigma$-equivalent class. Any two chains in the same $\delta$-equivalent class are generated by the same tensors so their reducing result is still in the same equivalent class. Thus each $\delta$-class forms a subalgebra of ${\mathcal{L}}(\Gamma)$ but a $\sigma$-class is impossible a subalgebra of $\mathcal{L}(\Gamma)$.

\subsection{Semilattices}

In this subsection we focus on $\delta$-classes. We recall that a semilattice is an idempotent commutative semigroup.

\begin{defn}\label{subsemilattice}
Let $Y$ be a quasi-semilattice. A a non-empty set $H$ of $Y$ is called a {\it sub-semilattice} of $Y$ if $H$ forms a semilattice with respect to the operation in $Y$.
\end{defn}

\begin{prop}\label{semilattice}
Each $\delta$-class of ${\mathcal{L}}(\Gamma)$ is a subsemilattcie of ${\mathcal{L}}(\Gamma)$ with respect to  $\mathring{\cup}$.
\end{prop}
\begin{proof}
Clearly, the $\delta$-class of  the empty element $\emptyset$ only contains $\emptyset$ so it is a semilattice with respect to $\mathring{U}$.

For any non-empty element $l \in {\mathcal{L}}(\Gamma)$, if $l = \oplus(P|S) \in {\mathcal{L}}(\Gamma)$, then according to the definition of $\delta$,  $l\delta$ consists of the tensor chains generated by $P$. Now we show that $l\delta$ is a semilattice with respect to $\mathring{\cup}$. It is sufficient to show that $\mathring{\cup}$ is closed and associative in $l\delta$ as $\mathring{\cup}$ is commutative and every element of ${\mathcal{L}}(\Gamma)$ is idempotent.  It is easy to see that $\mathring{\cup}$ is closed in $l\delta$ because for any $l_1, l_2 \in l\delta$ with $l_1 = \oplus(P_1|S_1)$ and $l_2 = \oplus(P_2|S_2)$ , we have $P_1 = P_2 = P$, and so $l_1 \mathring{\cup} l_2 = \oplus(P|S_1 \cup S_2) \in l\delta$. To show that $\mathring{\cup}$ is associative in $l\delta$, suppose that  $l_1 = \oplus(P|S_1)$, $l_2 = \oplus(P|S_2)$ and $l_3 = \oplus(P|S_3)$ are elements of $l\delta$. As their tensor sets of $l_1$, $l_2$ and $l_3$ are the same, that is $P$,  then the tensor set of their reducing result is still $P$, and also we have $(l_1 \mathring{\cup} l_2) \mathring{\cup} l_3 = \mathring{\cup}(S_1 \cup S_2) \mathring{\cup} l_3 = \mathring{\cup}(S_1 \cup S_2 \cup S_3) = l_1 \mathring{\cup} (\mathring{\cup}(S_2 \cup S_3)) = l_1 \mathring{\cup} (l_2 \mathring{\cup} l_3)$, which implies that $\mathring{\cup}$ is associative in $l\delta$. Consequently, $l\delta$ is a semilattice with respect to $\mathring{\cup}$.  According to Definition~\ref{subsemilattice} $l\delta$ is a subsemilattice of ${\mathcal{L}}(\Gamma)$.
\end{proof}

A $\sigma$-equivalence class consisting of $k$ tensors in ${\mathcal{L}}(\Gamma)$ is the set of all $k$-chains. Each $k$-chain only belongs to one $\delta$ class. Therefore, a $\sigma$-equivalence class with  $k$ tensors is the disjoint union of all $\delta$-equivalence classes with $k$ tensors, that is, the disjoint union of sub-semilattices. The relationship between different $\delta$-equivalence classes in the same $\sigma$-equivalence class is discussed below.

Let $l_1 = \oplus(P_1|S_1)$ and $l_2 = \oplus(P_2|S_2)$ be two non-empty elements of ${\mathcal{L}}(\Gamma)$  with $|P_1| = k_1$ and $|P_2| =k_2$, and let $S_{P_1}$ and $S_{P_2}$ be the set of all 2-chains generated by $P_1$ and $P_2$, respectively, that is,
$$S_{P_1} = \{l_{i_1j_1}|i_1, j_1 = 1, 2, \ldots, k_1; i_1 \neq j_1; l_{i_1j_1} \neq \emptyset\},$$
$$S_{P_2} = \{l_{i_2j_2}|i_2, j_2= 1, 2, \ldots, k_2; i_2 \neq j_2; l_{i_2j_2} \neq \emptyset\}.$$

\begin{prop}
If $|P_1|=|P_2|$, $|S_{P_1}|=|S_{P_2}|$ and there exist two bijections $\theta: P_1 \rightarrow P_2$ and $\tau: S_{P_1} \rightarrow S_{P_2}$ with $\tau(l_{i_1j_1}) = l_{\theta(i_1)\theta(j_1)} \in S_{P_2}$ then $l_1\delta$ and $l_2\delta$ are isomorphic semilattices with respect to $\mathring{\cup}$ where $l_1 = \oplus(P_1|S_1)$ and $l_2 = \oplus(P_2|S_2)$.
\end{prop}

\begin{proof}
It follows from Proposition~\ref{semilattice} that $l_1\delta$ and $l_2\delta$ are subsemilattices of ${\mathcal{L}}(\Gamma)$. Assume that $|P_1|=|P_2|$, $|S_{P_1}|=|S_{P_2}|$ and there exist two bijections $\theta: P_1 \rightarrow P_2$ and $\tau: S_{P_1} \rightarrow S_{P_2}$ with $\tau(l_{i_1j_1}) = l_{\theta(i_1)\theta(j_1)} \in S_{P_2}$. Let $\eta: l_1\delta \rightarrow l_2\delta$ be a map with for any $l =\oplus(P_1|S_l) \in l_1\delta$, $\eta(l) = \oplus(\theta(P_1)|\tau(S_l))$, where $\theta(P_1) = P_2$ and $\tau(S_1) = \{\tau(l_{i_1j_1})|l_{i_1j_1} \in S_1 \subseteq S_{P_2}\}$. Clearly, $\eta$ is well defined.

To show that $\eta$ is injective, we suppose that $l =\oplus(P_1|S_l)$ and $l' = \oplus(P_1|S_{l'})$ are two elements of $l_1\delta$ such that $\eta(l) = \eta(l')$. Then $\mathring{\cup}\tau(S_l) = \mathring{\cup}\tau(S_{l'})$. Further by Definition~\ref{two-chain reducing} and its generalization on tensor chains, we obtain that $\tau(S_l) = \tau(S_{l'})$. As $\tau: S_{P_1} \rightarrow S_{P_2}$ is a bijection it follows that $S_l = S_{l'}$, which implies that $l = l'$.

Next we show that $\eta$ is surjective. For any $l'' \in l_2\delta$ with $l'' = \oplus(P_2|S_{l''})$. Since $\tau: S_{P_1} \rightarrow S_{P_2}$ is a bijection, it follows that $\tau^{-1}(S_{l''}) \subseteq S_{P_1}$. As the tensor set generating the 2-chains in $S_{l''}$ is $P_2$ and $\theta: P_1 \rightarrow P_2$ is a bijection, we obtain that $P_1$ is the tensor set generating the 2-chains in $\tau^{-1}(S_{l''})$ and so $\mathring{\cup}\tau^{-1}(S_{l''}) \in l_1\delta$ such that $\eta(\mathring{\cup}\tau^{-1}(S_{l''})) = l''$.

Finally, we claim that $\eta$ is a homomorphism. For any two elements $l =\oplus(P_1|S_l)$ and $l' = \oplus(P_1|S_{l'})$ of $l_1\delta$, we have
$\eta(l_1 \mathring{\cup} l_2) = \eta(\oplus(P_1|S_l \cup S_{l'}) = \oplus(\theta(P_1)|\tau(S_l \cup S_{l'})) = \mathring{\cup} (\tau(S_l \cup S_{l'}) = \oplus(P_2|\tau(S_l)) \mathring{\cup} \oplus(P_2|\tau(S_{l'})) = \eta(l) \mathring{\cup} \eta(l')$.

To sum up, $\eta: l_1\delta \rightarrow l_2\delta$ is an isomorphism.
\end{proof}

$\delta$ is a congruence on ${\mathcal{L}}(\Gamma)$ and tensor chains in the same $\delta$-class consist of the same tensor set. So if  the intersection of the tensor sets of two different $\delta$-classes is non-empty then the reducing of two chains from the two different $\delta$-classes belongs to a $\delta$-class whose tensor set is the union of  tensor sets of the two different $\delta$-classes, otherwise, the reducing of two chains from the two different $\delta$-classes is the empty element $\emptyset$. It means that $\mathring{\cup}$ deduces a binary operation in the set of $\delta$-classes of ${\mathcal{L}}(\Gamma)$ as follows:

Let ${\mathcal{L}}(\Gamma)/\delta$ be the set of all the $\delta$-classes of ${\mathcal{L}}(\Gamma)$, that is ${\mathcal{L}}(\Gamma)/\delta = \{l\delta| l \in {\mathcal{L}}(\Gamma)\}$. For any two element $l_1\delta, l_2\delta \in {\mathcal{L}}(\Gamma)/\delta$,
\[(l_1\delta) * (l_2\delta) = (l_1 \mathring{\cup} l_2)\delta.\]

\begin{lem}
The set ${\mathcal{L}}(\Gamma)/\delta$ forms a quasi-semilattice with respect to $*$ defined above.
\end{lem}
\begin{proof}
We first show that $*$ is well defined. Suppose that $l_1\delta, l'_1\delta, l_2\delta, l'_2\delta \in {\mathcal{L}}(\Gamma)/\delta$ are such that $l_1\delta = l'_1\delta$ and $l_2\delta = l'_2\delta$. We can deduce that $(l_1\delta) * (l_2\delta) = (l_1 \mathring{\cup} l_2)\delta = (l'_1 \mathring{\cup} l_2)\delta  = (l'_1 \mathring{\cup} l'_2)\delta = (l'_1\delta) * (l'_2\delta)$ as $\delta$ is  left and right compatible. Thus $*$ is well defined.

It is easy to see that for any $l\delta \in {\mathcal{L}}(\Gamma)/\delta$, we have $(l\delta) * (l\delta) = (l \mathring{\cup} l)\delta = l\delta$. So each element of ${\mathcal{L}}(\Gamma)/\delta$ is idempotent. As $\mathring{\cup}$ is commutative we have for any  $l_1\delta, l_2\delta \in {\mathcal{L}}(\Gamma)/\delta$, $(l_1\delta) * (l_2\delta) = (l_1 \mathring{\cup} l_2)\delta = (l_2 \mathring{\cup} l_1)\delta = (l_2\delta) * (l_1\delta)$.

Finally we show that $*$ is not associative. Suppose that $l_1 = \oplus(P_1|S_1)$, $l_2 = \oplus(P_2|S_2)$ and $l_3 = \oplus(P_3|S_3)$ are elements of ${\mathcal{L}}(\Gamma)$ such that $P_1 \cap P_2 \neq \emptyset$, $P_1 \cap P_3 \neq \emptyset$, and $P_2 \cap P_3 = \emptyset$. Then
\[(l_1\delta * l_2\delta) * l_3\delta = (l_1 \mathring{\cup} l_2)\delta * l_3\delta = ((l_1 \mathring{\cup} l_2) \mathring{\cup} l_3) \delta = \mathring{\cup}(S_1 \cup S_2 \cup S_3) \neq \emptyset\delta,\]
\[l_1\delta * (l_2\delta * l_3\delta) = l_1\delta * (l_2 \mathring{\cup} l_3)\delta = l_1\delta * \emptyset\delta = (l_1 \mathring{\cup} \emptyset)\delta = \emptyset\delta,\]
which implies that $*$ is not associative.
\end{proof}

Naturally, there is a map from $({\mathcal{L}}(\Gamma), \mathring{\cup})$ to $({\mathcal{L}}(\Gamma)/\delta, *)$. In the following a homomorphism between two quasi-semilattices is used to describe the map.

\begin{defn}
Let $(Y, \cdot)$ and $(Y', \circ)$ be two quasi-semilattice, and let $\xi$ is a map from $Y$ to $Y'$.  The map $\xi$ is a {\it homomorphism} if for any $a, b \in Y$, $\xi(a \cdot b) = \xi(a) \circ \xi(b)$. A homomorphism $\xi$ is an {\it epimorphism} if $\xi$ is surjective.
\end{defn}

\begin{them}
The map $\chi: {\mathcal{L}}(\Gamma) \rightarrow {\mathcal{L}}(\Gamma)/\delta$ with $\chi(l) = l\delta$ for any $l \in {\mathcal{L}}(\Gamma)$
 is an epimorphism.
 \end{them}
\begin{proof}
It is clear that $\chi$ is surjective . For any $l_1, l_2 \in {\mathcal{L}}(\Gamma)$, we have $\chi(l_1 ) * \chi(l_2)= l_1\delta * l_2\delta = (l_1 \mathring{\cup} l_2)\delta = \chi(l_1 \mathring{\cup} l_2)$. Thus $\chi$ is a homomorphism.
\end{proof}

At the end of this section, we remark that given a network $\Gamma$, ${\mathcal{L}}(\Gamma)$ consists of all the tensor chains and  all the $\delta$-classes form a partition of ${\mathcal{L}}(\Gamma)$. Each $\delta$-class is a semilattice.  ${\mathcal{L}}(\Gamma)$ is a joint union of semilattices.  Each semilattice contains all the model of the tensor chains generated by the same tensor set.

\section{Order relations}\label{sec:orderrelation}

The $\mathring{\cup}$ operation on ${\mathcal{L}}(\Gamma)$ can induce an order relation $\preceq$ on ${\mathcal{L}}(\Gamma)$ defined below. For any $l_1, l_2 \in {\mathcal{L}}(\Gamma)$, we define that
\[l_1 \preceq l_2\ \mbox{if}\ l_1 \mathring{\cup} l_2 = l_2. \]

\begin{prop}\label{order}
The relation $\preceq$ defined above is a partial order relation on $({\mathcal{L}}(\Gamma), \mathring{\cup})$.
\end{prop}
\begin{proof}
$\rm (i)$ Clearly, $\preceq$ is reflective since for any $l \in {\mathcal{L}}(\Gamma)$, we have $l \mathring{\cup} l = l$.

$\rm(ii)$ To show that $\preceq$ is anti-symmetric, assume that $l_1, l_2 \in {\mathcal{L}}(\Gamma)$ are such that $l_1 \preceq l_2$ and $l_1 \neq l_2$. Then $l_1 \mathring{\cup} l_2 = l_2$. If $l_2 \preceq l_1$, we have $l_2 \mathring{\cup} l_1 = l_1$. It follows from $\mathring{\cup}$ being commutative that $l_1 = l_2 \mathring{\cup} l_1 = l_1 \mathring{\cup} l_2 = l_2$ which is a contradiction with $l_1 \neq l_2$. So $l_2 \npreceq l_1$.

$\rm (iii)$ Now we show that $\preceq$ is transitive. Suppose that $l_1, l_2, l_3 \in {\mathcal{L}}(\Gamma)$  are such that $l_1 \preceq l_2$ and $l_2 \preceq l_3$. Let $l_1 = \oplus(P_1|S_1)$, $l_2 = \oplus(P_2|S_2)$ and $l_3 = \oplus(P_3|S_3)$. By $l_1 \preceq l_2$ we get that $P_1 \subseteq P_2$ and $S_1 \subseteq S_2$. Also by $l_2 \preceq l_3$ we get that $P_2 \subseteq P_3$ and $S_2 \subseteq S_3$. Then we have $P_1 \subseteq P_3$ and $S_1 \subseteq S_3$ and further we obtain that $l_1 \mathring{\cup} l_3 = \mathring{\cup}(S_1 \cup S_3) = l_3$, that is, $l_1 \preceq l_3$.
\end{proof}

Next, we discuss the local maximum and the local minimum elements in a $\delta$-class by using the partial order relation $\preceq$ on ${\mathcal{L}}(\Gamma)$ defined above.

\begin{defn}
A non-empty element $l \in {\mathcal{L}}(\Gamma)$ is called a {\it local maximum element} of $l\delta$ if for any $l' \in l\delta$ we have $l' \preceq l$. A non-empty element $l \in {\mathcal{L}}(\Gamma)$ is called a {\it local minimum element} of $l\delta$ if for any $l' \in l\delta$ with $l' \preceq l$ we get $l' = l$.
\end{defn}

\begin{prop}
Each $\delta$-class of ${\mathcal{L}}(\Gamma)$ contains a unique local maximum element.
\end{prop}
\begin{proof}
If $\delta$-class contains only $\emptyset$ then it is clear that $\emptyset$ is the unique local maximum element of $\emptyset\delta$. Let $l \in {\mathcal{L}}(\Gamma)$ be a non-empty chain with $l = \oplus(P|S)$. Then every chain in $l\delta$ is generated by $P$. Let $S_P$ denote the set of all 2-chains generated by $P$. Then for any $l' \in l\delta$ with $l' = \oplus(P|S_{l'})$, we have $S_{l'} \subseteq S_P$ and also $l' \mathring{\cup} (\mathring{\cup}S_P) = \mathring{\cup}(S_{l'} \cup S_P) = \mathring{\cup}S_P \in l\delta$, which follows that $l' \preceq \mathring{\cup}S_P $ and so $\mathring{\cup}S_P$ is the local maximum element.

To show the unique, suppose that  $l'' = \oplus(P|S_{l''}) \in l\delta$ is another local maximum element. Then  $\mathring{\cup}S_P \preceq l''$, that is, $l'' \mathring{\cup} (\mathring{\cup}S_P) = l''$ which means that $S_P \cup S_{l''} = S_{l''}$. Together with $S_{l''} \subseteq S_P$ we get $S_{l''} = S_P$, and so $l'' = \mathring{\cup}S_P$.
\end{proof}

Before we show that each $\delta$-class of ${\mathcal{L}}(\Gamma)$ contains local minimum elements, we recall spanning trees. The {\it spanning tree } in a connected undirected network must meet the following two conditions:(i)it contains all vertices in a connected grapg; (ii)there is only one path between any two vertices.

\begin{prop}
Each $\delta$-class of ${\mathcal{L}}(\Gamma)$ contains local minimum elements.
\end{prop}
\begin{proof}
If $\delta$-class contains only $\emptyset$ then it is clear that $\emptyset$ is the unique local minimum element of $\emptyset\delta$. Let $l \in {\mathcal{L}}(\Gamma)$ be a non-empty chain with $l = \oplus(P|S)$, where $P = \{t_1, t_2, \ldots, t_k\}.$ Then each chain in $l\delta$ is generated by $P$. Notice that chains are connected and connected graphs have spanning trees. So $l\delta$ has spanning trees generated by $P$. Assume that $l_T$ is a spanning tree generated by $P$. If $l' = \oplus(P|S_{l'}) \in l\delta$ is such that $l' \preceq l_T$. Then $l' \mathring{\cup} l_T = l_T$, which follows that $S_{l'} \subseteq S_{l_T}$. Since $l_T$ is a spanning tree it contains only $k-1$ edges, that is $|S_{l-T}| = k-1$. So $|S_{l'}|\leq k-1$. But $l'$ is a connected chain containing all the tensors in $P$ so $|S_{l'}|\geq k-1$. Thus we obtain that $|S_{l'}|= k-1$. Hence a spanning tree $l_T$ is a local minimum element.
\end{proof}

Let $l \in {\mathcal{L}}(\Gamma)$ be a non-empty chain with $l = \oplus(P|S)$, where $P = \{t_1, t_2, \ldots, t_k\}.$ Then each chain in $l\delta$ is generated by $P$. Let $S_P$ denote the set of all 2-chains generated by $P$. Then we can draw a undirected graph  whose sets of vertices and edges are $P$ and $S_P$, respectively. Its Laplace matrix is denoted by $A$. $A$ is a $k \times k$ matrix whose element in the $i$-th row and $j$-th column is that
$$a_{ij} =
\begin{cases}
\ -1 & \mbox{if } \ i \neq j \ \mbox{and}\ l_{ij}\neq \emptyset\\
\ \mbox{deg}\ (t_i) & \mbox{if}\ i = j,
\end{cases}
$$
where deg$(t_i)$ is the number of 2-chains generated by $t_i$, that is the degree of $t_i$. Furthermore, from the matrix-tree theorem \cite{Boccaletti4}, it can be obtained that the number of local minimum elements in the $l\delta$ class is equal to det$A_0$, where $A_0$ is the submatrix removing the $i$-th row and $i$-th column of Laplace matrix $A$, where $i$ is an arbitrary element of the set $\{1, 2, \ldots, k\}$.

\section{Path algebras}\label{sec:path}

In this section three path algebras: graph inverse semigroups, Leavitt path algebra and Cuntz-Krieger graph $C^*$-algebra are constructed in terms of  tensors with respect to chain-addition. Since the three path algebras are based on first-order directed networks, only first-order directed networks, referred to as directed networks, are considered in this section.

Ash and Hall \cite{Ash8} first introduced graph inverse semigroups by using the vertices and edges of directed networks. Graph inverse semigroups are a generalization of polycyclic monoids introduced by Nivart and Perrot \cite{Nivat15}. Let $\Gamma$  be a directed network. $\Gamma^{0}$  denotes the set of vertices of $\Gamma$, $\Gamma^1$ denotes the set of edges of $\Gamma$. For any $e \in \Gamma^1$, $s(e)$ is the {\it source} of $e$ and $r(e)$ is the {\it range} of $e$. For any $v \in \Gamma^0$, we regard $v$ as a path of length $0$ (or a trivial path), $s(v) = r(v) = v$, $s^{-1}(v)= \{e \in \Gamma^1: s(e) = v\}$ and the out-degree of a vertex $v$ is $|s^{-1}(v)|$, the number of directed edges with source $v$.   For any $e \in \Gamma^1$, $e^*$ is the {\it inverse} of $e$ such that $s(e^*) = r(e)$ and $r(e^*) = s(e)$. Let $(\Gamma^1)^*$ be the set $\{e^*| e \in \Gamma^1\}$. Assume that $\Gamma^1 \cap (\Gamma^1)^* = \emptyset$. The {\it graph inverse semigroup} $I(\Gamma)$ of $\Gamma$ is the semigroup with a zero element $o$ generated by $\Gamma^0 \cup \Gamma^1 \cup (\Gamma^1)^*$, satisfying the following relations :

\begin{enumerate}
\item[(I1)] for all $e \in \Gamma^0 \cup \Gamma^1 \cup (\Gamma^1)^*$, $s(e)e = er(e) = e$;
\item[(I2)] for any $u, v \in \Gamma^0$ with $u \neq v$, $uv = o$;
\item[(I3)] for any $e, f \in \Gamma^1$ with $e \neq f$, $e^*f = o$;
\item[(I4)] for $e \in \Gamma^1$, $e^*e = r(e)$.
\end{enumerate}

Next, we discuss how to generate graph inverse semigroups from tensors.

Let $e$ be an arbitrary element of $\Gamma^1$.  It is natural to write $e$,$s(e)$ and $r(e)$ as a tensor, denoted it by $e^{r(e)}_{s(e)}$. Let $T= \{e^{r(e)}_{s(e)}| e \in \Gamma^1\}$, where $e^{r(e)}_{s(e)}$ is the tensor with contra-variant index  $r(e)$ and  covariant index $s(e)$. Clearly it corresponds to  the edge $e$. Without ambiguity, it is abbreviated as $e$. According to the definition of graph inverse semigroups, the generators have not only the edges of directed graph, but also the inverse of edge and empty directed path at vertices. Therefore, let $T^0 = \{v^v_v|v \in \Gamma^0\}$, where $v^v_v$ denotes the empty tensor at $v$. Without ambiguity, it is denoted by $v$. Let $T^* = \{e^{s(e)}_{r(e)}| e \in \Gamma\}$, where $e^{s(e)}_{r(e)}$ denotes the tensor with contra-variant index  $s(e)$ and  covariant index $r(e)$, which corresponds to the inverse $e^*$ of $e$. Without ambiguity, it is denoted by $e^*$.

A {\it directed path} in $\Gamma$ is a finite sequence $p = e_1e_2\ldots e_k$ of edges $e_i \in \Gamma^1$ with $r(e_i) = s(e_{i+1})$ for $i =1, \ldots, k-1$. We define $s(p) = s(e_1)$ and $r(p) = r(e_k)$. For any two different tensors $e_1$ and $e_2$ with $r(e_1) = s(e_2)$, the 2-chain $\oplus(e_1, e_2)$ generated by $e_1$ and $e_2$ is a directed path from $s(e_1)$ to $r(e_2)$, called a {\it 2-line chain}, simply denoted by $\oplus\langle e_1, e_2\rangle$. Here $\langle e_1, e_2 \rangle$ reflects the direction of the chain $\oplus\langle e_1, e_2\rangle$, that is from $s(e_1)$ to $r(e_2)$. So $\oplus \langle e_1, e_2\rangle \neq \oplus\langle e_2, e_1 \rangle$.

For any directed path $p = e_1e_2 \cdots e_k$ with $e_i \in \Gamma^1$ and $r(e_i) = s(e_{i+1})$ for $i = 1, 2, \cdots, k-1$, it is written as
\[\oplus\langle e_1, e_2\rangle \mathring{\cup} \oplus\langle e_2, e_3\rangle \mathring{\cup} \cdots \mathring{\cup} \oplus\langle e_{k-1}, e_k\rangle\]
in view of the reducing of 2-chains. That is, the adjacent same tensors in $\oplus\langle e_1, e_2\rangle \oplus\langle e_2, e_3\rangle \cdots\cdots \oplus\langle e_{k-1}, e_k\rangle$ are combined into one, and the rest remain unchanged. Hence the path $p = e_1e_2 \cdots e_k$ is a tensor chain generated by tensors $e_1, e_2, \ldots, e_k$ according to the 2-chain set $S = \{\oplus\langle e_i, e_{i+1}\rangle| i = 1, 2, \ldots, k-1\}$, and the contra-variant index of the tensor in $p$ is the covariant index of its subsequent tensor.

\begin{defn}
Let $P = \{e_1, e_2, \cdots, e_k \}$, $S = \{\oplus\langle e_i, e_{i+1}\rangle| i = 1, 2, \ldots, k-1\}$. The tensor chain $\oplus(P|S)$ is called a {\it directed path} if $\varphi(e_i) = \psi(e_{i+1})$, $i = 1, 2, \cdots, k-1$. $\oplus(P|S)$ is simply denoted by $e_1e_2\cdots e_k$.
\end{defn}

Only 2-line chains are considered in the directed path. Therefore, in order to generate the directed path and further generate graph inverse semigroups, the conditions for the chain addition of two different tensors in Definition~\ref{ChainAdd}  is extended to $T^0 \cup T \cup T^*$, as follows:

\begin{defn} \label{ChainAdd3}
For any $t, t_1, t_2 \in T^0 \cup T \cup T^*$,
$$t_1 \overline{\oplus} t_2 =
\begin{cases}
\ t_1 & \mbox{if } \ t_1 \in T^0 \cup T \cup T^*,\ t_2 \in T^0\ \mbox{and}\ \varphi(t_1) = \psi(t_2)\\
\ t_2 & \mbox{if } \  t_1 \in T^0,\ t_2 \in T^0 \cup T \cup T^*\ \mbox{and}\ \varphi(t_1) = \psi(t_2)\\
\ t_1t_2 & \mbox{if } \ t_1, t_2 \in T\ \mbox{or}\ t_1, t_2 \in T^* \ \mbox{or}\ t_1 \in T, t_2 \in T^* \ \mbox{and}\ \varphi(t_1) = \psi(t_2)\\
\ \varphi(t_2) & \mbox{if } \ t_1 \in T^*, t_2 \in T, t_1^* = t_2\\
\ \emptyset & \mbox{otherwise},
\end{cases}
$$
$$\emptyset\overline{\oplus} t = t \overline{\oplus} \emptyset = \emptyset,$$
where $t_1t_2 = \oplus\langle t_1, t_2\rangle$.
\end{defn}

Obvilusly, $\overline{\oplus}$ can be generalised to the directed paths. For any $t_1, t_2, t_3 \in T^0 \cup T \cup T^*$, if $t_1 \overline{\oplus} t_2 = t_1t_2$ then we define $t_1t_2 \overline{\oplus} t_3 = t_1(t_2 \overline{\oplus} t_3)$. So for any directed paths $t_{i_1}t_{i_2} \cdots t_{i_{m_1}}$ and $t_{j_1}t_{j_2} \cdots t_{j_{m_2}}$, we have
\[t_{i_1}t_{i_2} \cdots t_{i_{m_1}} \overline{\oplus} t_{j_1}t_{j_2} \cdots t_{j_{m_2}} = t_{i_1}t_{i_2} \cdots t_{i_{m_1-1}} (\cdots((t_{i_m} \overline{\oplus}t_{j_1})\overline{\oplus}t_{j_2})\overline{\oplus} \cdots \overline{\oplus}t_{j_{m_2}}).\]

\begin{lem}
The semigroup $T(\Gamma)$ generated by $T^0 \cup T \cup T^*$ and empty element $\emptyset$  together with $\overline{\oplus}$ is an inverse semigroup.
\end{lem}
\begin{proof}
According to Definition~\ref{ChainAdd3}, it is easy to see that the forms of elements of $T(\Gamma)$ are $p, p^*, pq^*$ where $p$ and $q$ are directed path generated by tensors in $T^0 \cup T$ such that $r(p) = r(q)$.  Because $r(p)^{r(p)}_{r(p)} \in T^0$, $(r(p)^{r(p)}_{r(p)})^* = r(p)^{r(p)}_{r(p)}$ and so by Definition~\ref{ChainAdd3} $p = p \overline{\oplus} r(p)^{r(p)}_{r(p)}$ and $p^* = r(p)^{r(p)}_{r(p)}p^*$. It follows that there is only one element form in $T$, that is $pq^*$, where $p$ and $q$ are directed path generated by tensors in $T^0 \cup T^1$ such that $r(p) = r(q)$.

For any $pq^*, xy^* \in T(\Gamma)$, we have
$$pq^* \overline{\oplus} xy^* =
\begin{cases}
\ pzy^* & \mbox{if } \ x = qz\\
\ p(yw)^* & \mbox{if } \ q = xw\\
\ \emptyset & \mbox{otherwise},
\end{cases}
$$
where $z$ and $w$ are directed path generated by tensors in $T^0 \cup T$.

Next, we show that $\oplus$ is associative in $T(\Gamma)$. Suppose that $pq^*, xy^*, gh^* \in T(\Gamma)$.

Case I. If $x = qz_1, g = yz_2$, then we have $(pq^* \overline{\oplus} xy^*) \overline{\oplus} gh^*  = pz_1y^* \overline{\oplus} gh^* = pz_1z_2h^*$, and $pq^* \overline{\oplus} (xy^* \overline{\oplus} gh^*) = pq^* \overline{\oplus} xz_2h^* = pq^*qz_1z_2h^* = pz_1z_2h^* $, which follows that $(pq^* \overline{\oplus} xy^*) \overline{\oplus} gh^*  = pq^* \overline{\oplus} (xy^* \overline{\oplus} gh^*)$.

Case II. If $x = qz_1, y = gw_2$, then $(pq^* \overline{\oplus} xy^*) \overline{\oplus} gh^*  = pq^*qz_1y^* \overline{\oplus} gh^* = pz_1(gw_2)^*gh^* = pz_1(hw_2)^*$ and  $pq^* \overline{\oplus} (xy^* \overline{\oplus} gh^*) = pq^* \overline{\oplus} x(gw_2)^*gh^* = pq^* \overline x(hw_2)^* = pq^*qz_1(hw_2)^* = pz_1(hw_2)^*$, which follows that $(pq^* \overline{\oplus} xy^*) \overline{\oplus} gh^*  = pq^* \overline{\oplus} (xy^* \overline{\oplus} gh^*)$.

Case III. If $q = xw_1, g = yz_2$ then $(pq^* \overline{\oplus} xy^*) \overline{\oplus} gh^*  = p(xw_1)^*xy^* \overline{\oplus} gh^* = p(yw_1)^*yz_2h^* = pw_1^*z_2h^*$ and $pq^* \overline{\oplus} (xy^* \overline{\oplus} gh^*) = pq^* \overline{\oplus} xy^*yz_2h^* = p(xw_1)^*xz_2h^* = pw_1^*z_2h^*$, which follows that $(pq^* \overline{\oplus} xy^*) \overline{\oplus} gh^*  = pq^* \overline{\oplus} (xy^* \overline{\oplus} gh^*)$.

Case IV. If $q = xw_1, y = gw_2$ then $(pq^* \overline{\oplus} xy^*) \overline{\oplus} gh^*  = p(gw_2w_1)^*gh^* = pw^*_1w^*_2h^* = p(hw_1w_2)^*$ and $pq^* \overline{\oplus} (xy^* \overline{\oplus} gh^*) = pq^* \overline{\oplus}x(gw_2)^*gh^* = pq^* \overline{\oplus} x(hw_2)^* = p(xw_1)^*x(hw_2)^* = p(hw_2w_1)^*$, which follows that $(pq^* \overline{\oplus} xy^*) \overline{\oplus} gh^*  = pq^* \overline{\oplus} (xy^* \overline{\oplus} gh^*)$.

Case V. Otherwise, we have $(pq^* \overline{\oplus} xy^*) \overline{\oplus} gh^*  = pq^* \overline{\oplus} (xy^* \overline{\oplus} gh^*) = \emptyset$. Thus, $\oplus$ is associative in $T(\Gamma)$.

Now we show that the set of idempotents of $T(\Gamma)$ is a semilattice with respect to $\overline{\oplus}$. It is easy to see that for any $pq^* \in T(\Gamma)$, $pq^* \overline{\oplus} pq^* = pq^*$ if and only if $q^*p = r(p)^{r(p)}_{r(p)}$ , if and only if $p = q$. So the set of idempotents of $T(\Gamma)$ is
 $$E(T(\Gamma)) = \{pp^* |\ p \ \mbox{is\ a\ directed\ path\ generated\ by\ tensors\ in}\ T^0 \cup T\} \cup \{\emptyset\}.$$
For any $pp^*, qq^* \in E(T(\Gamma))$, we have
$$pp^* \overline{\oplus} qq^* =
\begin{cases}
\ pp^* & \mbox{if } \ p = qz\\
\ qq^* & \mbox{if } \ q = pw\\
\ \emptyset & \mbox{otherwise},
\end{cases}
$$
and also $pp^* \overline{\oplus} qq^* = qq^* \overline{\oplus} pp^*$. Hence $E(T(\Gamma))$ is a semilattice.

Finally, for any $pq^* \in T(\Gamma)$, $qp^*$ is an inverse of $pq^*$ since $pq^* \overline{\oplus} qp^* \overline{\oplus} pq^* = pq^*$ and $qp^* \overline{\oplus} pq^* \overline{\oplus} qp^* = qp^*$.

To sum up, $T(\Gamma)$ is an inverse semigroup.
 \end{proof}

 \begin{prop}\label{TwoInv}
  The map $\theta: I(\Gamma) \rightarrow T(\Gamma)$ , defined by the rule that \\
 {\rm (i)} for any $v \in \Gamma^0$, $\theta(v) = v^v_v$;\\
  {\rm (ii)} for any $e \in \Gamma^1$, $\theta(e) = e^{r(e)}_{s(e)}$;\\
 {\rm (iii)} for any $e^* \in (\Gamma^1)^*$, $\theta(e^*) = e^{s(e)}_{r(e)}$;\\
 {\rm (iv)} $\theta(o) = \emptyset$,\\
 is an isomorphism.
 \end{prop}
\begin{proof}
We first show that conditions (I1)-(I4) in the definition of graph inverse semigroups hold.

(I1) For any $e \in \Gamma^0 \cup \Gamma^1 \cup (\Gamma^1)^*$, we have $\theta(s(e)) \overline{\oplus}\theta(e) = s(e)^{s(e)}_{s(e)} \overline{\oplus} e^{r(e)}_{s(e)} = e^{r(e)}_{s(e)} = \theta(e)$ and $\theta(e) \overline{\oplus} \theta(r(e)) = e^{r(e)}_{s(e)} \overline{\oplus} r(e)^{r(e)}_{r(e)} = e^{r(e)}_{s(e)} = \theta(e)$.

(I2) For any $u, v \in \Gamma^0$ with $u \neq v$, $\theta(u) \overline{\oplus} \theta(v) = u^u_u \overline{\oplus} v^v_v = \emptyset = \theta(o)$.

(I3) If $e, f \in \Gamma^1$ are such that $e \neq f$, then $\theta(e^*) \overline{\oplus} \theta(f) = e^{s(e)}_{r(e)} \overline{\oplus} f^{r(f)}_{s(f)} = \emptyset = \theta(o)$.

(I4) If $e \in \Gamma^1$, then $\theta(e^*) \overline{\oplus} \theta(e) = e^{s(e)}_{r(e)} \overline{\oplus} e^{r(e)}_{s(e)} = r(e)^{r(e)}_{r(e)} = \theta(r(e))$.

Let $W = \Gamma^0 \cup \Gamma^1 \cup (\Gamma^1)^*$. Then $\theta$ is a homomorphism from $I(\Gamma) =F_W/\sim$ to $T(\Gamma)$, where $F_W$ is the free semigroup over $W$  and $\sim$ is the relation satisfying (I1)-(I4).

To show that $\theta$ is surjective, suppose that $p = e_1{^{r(e_1)}_{s(e_1)}}e_2{^{r(e_2)}_{r(e_1)}}\cdots e_k{^{r(e_k)}_{r(e_{k-1})}}$, $q = f_1{^{r(f_1)}_{s(f_1)}}f_2{^{r(f_2)}_{r(f_1)}}\cdots f_m{^{r(f_m)}_{r(f_{m-1})}}$ and $r(e_k) = r(f_m)$ then

$$pq^* = e_1{^{r(e_1)}_{s(e_1)}}e_2{^{r(e_2)}_{r(e_1)}} \cdots e_k{^{r(e_k)}_{r(e_{k-1})}}f_m{^{r(f_m)}_{r(f_{m-1})}}  \cdots f_2{^{r(f_2)}_{r(f_1)}}f_1{^{r(f_1)}_{s(f_1)}}.$$
Further  we have
$$\theta(e_1e_2\cdots e_kf_m^* \cdots f_2^*f_1^*) = \theta(e_1)\overline{\oplus}\theta(e_2) \overline{\oplus} \cdots \overline{\oplus} \theta(e_k) \overline{\oplus} \theta(f_m^*) \overline{\oplus} \cdots \overline{\oplus} \theta(f_2^*) \overline{\oplus} \theta(f_1^*) = pq^*.$$
Hence $\theta$ is surjective. Obviously it is injective and so $\theta$ is an isomorphism.
\end{proof}

Graph inverse semigroups are also an extension of graph $C^*$-algebra\cite{Kumjian9} and Leavitt path algebras\cite{Abrams10, Ara11}. Leavitt path algebras is derived from W.G. Leavitt's seminal paper\cite{Leavitt16}. Suppose that $F$ is a field and the out-degree of each vertex in a directed network $\Gamma$ is finite. The $F$-algebra generated by $\Gamma^0 \cup \Gamma^1 \cup (\Gamma^1)^*$ and a zero element $o$ is a {\it Leavitt path algebra} $L_F(\Gamma)$ if it satisfies relations (I1)-(I4) and the following Cuntz-Krieger relations (CK1):

(CK1) $v = \sum_{e \in s^{-1}(v)}ee^*$ for any $v \in \Gamma^0$ and the out-degree of $v$ is greater than $0$.

If $F$ is the complex field $\mathcal{{C}}$, Leavitt path algebra and Cuntz-Krieger graph $C^*$-algebra\cite{Kumjian9} are close related. Suppose that $\Gamma^0$ and $\Gamma^1$ are countable  and the out-degree of each vertex of $\Gamma$ is finite. The universal $C^*$-algebra is generated by $\Gamma^0 \cup \Gamma^1 \cup (\Gamma^1)^*$ and a zero element $o$ is a {\it Cuntz-Krieger graph } $C^*$ -{\it algebra} $C^*(\Gamma)$ if it satisfies relations (I1)-(I4), (CK1) and the following (CK2):

(CK2) for any $e \in \Gamma^1$, $ee^* \leq s(e)$.

 \noindent Here for any $u, v \in \Gamma^0$, $u \leq v$ if there exists a path $p$ such that $s(p) = v$ and $r(p) = u$; for any path $q = e_1e_2 \ldots e_k$ and $v \in \Gamma^0$, $q \leq v$ if $r(e_i) \leq v$ for all $i \in {1,2, \ldots, k}$.

Obviously, $T^0$, $T$ and $T^*$ are corresponding to $\Gamma^0$, $\Gamma^1$ and $(\Gamma^1)^*$ one to one , respectively. Due to the proof of Proposition~\ref{TwoInv} the chain addition on $T^0$, $T$ and $T^*$  satisfies (I1)-(I4), and so according to the definition of Leavitt path  algebra and Cuntz-Krieger graph $C^*$-algebra, except that the directed graph satisfies certain conditions (the out-degree of each vertex is finite, and the sets of vertices and edges are countable), only (CK1) and (CK2) need to be satisfied. Therefore, based on Definition~\ref{ChainAdd3}, if the following (TCK1) and (TCK2) are satisfied,  Leavitt path algebra and Cuntz Krieger graph $C^*$-algebra  can be generated by $T^0$, $T$ and $T^*$ and $\emptyset$ with respect to  the chain addition $\overline{\oplus}$.

(TCK1) $v^v_v = \sum_{\psi(e) = v}ee^*$ for any $v^v_v \in T^0$ with  the out-degree of $v$ is greater than $0$;

(TCK2) for any $e \in T$, $ee^* \leq \psi(e)$.

 Graph inverse semigroups,  Leavitt path algebra and Cuntz-Krieger graph $C^*$-algebra are algebraic systems of directed paths generated by tensors. Network quasi-semilattice  ${\mathcal{L}}(\Gamma)$ contains all the subnetworks  of $\Gamma$. Therefore, the study of the algebraic structure of quasi semilattice ${\mathcal{L}}(\Gamma)$  will help to study all the substructures (including paths) of the network.



\begin{thebibliography}{99}
\bibitem{G1} G. Stefanescu, Network algebra, Springer-Verlag London, 2000.
\bibitem{Peter2}P.M. Bower, S.J. Cokus, D. Eisenberg and T.O. Yeates, Use of Logic Relationships to Decipher Protein Network Organization, Science 306 (2004) 2246-2249.
\bibitem{Newman3} M.E.J. Newman, The structure and function of complex networks, SIAM Rev. 45 (2003) 167-256.

\bibitem{Boccaletti4}  S. Boccaletti,  V. Latora,  Y. Moreno, et al., Complex networks: Structure and dynamics,  Physics reports 424 (2006) 175-308.
\bibitem{Newman5}M.E.J. Newman, Finding community structure in networks using the eigenvectors of matrices, Physical review E 74 (2006) 036104.
\bibitem{Jean6}J. Piaget, Le Structuralisme, Presses Universitaires de France, 1983.
\bibitem{Bernhard7}B. Palsson, Systems Biology-Properties of Reconstructed Networks, Cambridge University Press, 2006.
\bibitem{Ash8}C.J. Ash, T.E. Hall, Inverse semigroups on graphs, Semigroup Forum 11 (1975) 140-145.
\bibitem{Kumjian9}A. Kumjian, D. Pask, I. Raeburn, Cuntz-Krieger algebras of directed graphs, Pac. J. Math. 184 (1998) 161-174.
\bibitem{Abrams10}G. Abrams, P. Aranda Pino, The Leavitt path algebra of a graph, J. Algebra 293 (2005) 319-334.
\bibitem{Ara11}P. Ara, M.A. Morene, E. Pardo, Nonstable K-theory for graph algebras, Algebr. Represent. Theory 10(2) (2007) 157-178.
\bibitem{Hala12}H.O. Pflugfelder, Quasi-groups and Loops: Introduction, Heldermann Verlag Berlin, 1990.
\bibitem{http13}https:$//$ en.wikipedia.org$/$wiki$/$Commutative-magma
\bibitem{Kirchhoff14}G. Kirchhoff, Uber die Auflfl$\ddot{o}$sung der Gleichungen, auf welche man bei der Untersuchung der linearen Verteilung galvanischer Str$\ddot{o}$me gefuhrt wird, Ann. Physik Chemie, 72 (1847) 497-508.
\bibitem{Nivat15}M. Nivat, J-F. Perrot, Une generalisation du monoide bicclique,  C. R. Acad. Sci. Paris 271 (1970) 824-827.
\bibitem{Leavitt16}W.G. Leavitt, The module type of a ring, Trans. Am. Math. Soc. 103 (1962) 113-130.

\end{thebibliography}
\end{document}